\documentclass[a4paper]{article}
\usepackage[utf8]{inputenc}
\usepackage{amssymb,amsthm,amsmath, amsfonts}
\usepackage{enumerate}

\title{The unconditional case of the complex $S$-inequality}

\author{Piotr Nayar \thanks{Research partially supported by NCN Grant no. 2011/01/N/ST1/01839.}, Tomasz Tkocz \thanks{Research partially supported by NCN Grant no. 2011/01/N/ST1/05960.} }
\date{}

\newtheorem{thm}{Theorem}
\newtheorem{lm}{Lemma}
\newtheorem{prop}{Proposition}
\newtheorem{cor}{Corollary}
\newtheorem*{conj*}{Conjecture}

\theoremstyle{definition}
\newtheorem{remark}{Remark}

\newcommand{\C}{\mathbb{C}}
\newcommand{\R}{\mathbb{R}}
\newcommand{\fun}[3]{#1\colon #2 \longrightarrow #3}

\newcommand{\dd}{\mathrm{d}}
\newcommand{\1}{\textbf{1}}
\newcommand{\mf}[1]{\mathfrak{#1}}

\DeclareMathOperator{\Ent}{Ent}

\begin{document}

\maketitle

\begin{abstract}
In this note we prove the complex counterpart of the S-inequality for complete Reinhardt sets. In particular, this result implies that the complex S-inequality holds for unconditional convex sets.
\end{abstract}



\section{Introduction}\label{sec.intro}

Studying various aspects of a Gaussian measure in a Banach space one often needs precise estimates on measures of balls and their dilations. This gives raise to the question how the function $(0, \infty) \ni t \mapsto \mu(tB)$ behaves. Here $B$ is a convex and symmetric subset of some Banach space, i.e. an unit ball with respect to some norm, and $\mu$ is a Gaussian measure. Thanks to certain approximation arguments we may only deal with the simplest spaces, namely $\R^n$ or $\C^n$. In the former case the issue is well understood due to R. Lata\l a and K. Oleszkiewicz. Denote by $\gamma_n$ the standard Gaussian measure on $\R^n$, i.e. the measure with the density at a point $(x_1, \ldots, x_n)$ equal to $\frac{1}{\sqrt{2\pi}^{n}}\exp\left( -x_1^2/2 - \ldots - x_n^2/2 \right)$. In \cite{LO} it is shown that for a symmetric convex body $K \subset \R^n$ and the strip $P = \{x \in \R^n \ | \ |x_1| \leq p\}$, where $p$ is chosen so that $\gamma_n(K) = \gamma_n(P)$, we have
\[
	\gamma_n(tK) \geq \gamma_n(tP), \qquad t \geq 1.
\]
This result is called \emph{S-inequality}. The interested reader is also referred to a concise survey \cite{Lat}.

In the present note we would like to focus on S-inequality for sets which correspond to unit balls with respect to unconditional norms on $\C^n$. Some partial results concerning general case has been recently obtained in \cite{Tko}.

Definitions and preliminary statements are provided in Section \ref{sec.pre}. Section \ref{sec.res} is devoted to the main result. It also contains a proof of a one-dimensional inequality, which bounds entropy, and seems to be the heart of the proof of our main theorem.

\section{Preliminaries}\label{sec.pre}

We define the standard Gaussian measure $\nu_n$ on the space $\C^n$ via the formula
\[
	\nu_n(A) = \gamma_{2n}\left( \tau (A) \right), \qquad \textrm{for any Borel set $A \subset \C^n$},
\]
where $\C^n \overset{\tau}{\longmapsto} \R^{2n}$ is the bijection given by 
\[
	\tau(z_1, \ldots, z_n) = (\mathfrak{Re} z_1, \mathfrak{Im} z_1, \ldots, \mathfrak{Re} z_n, \mathfrak{Im} z_n).   
\]
We adopt the notation $\R_+ = [0, +\infty)$. Later on we will also extensively use the notion of the \emph{entropy} of a function $\fun{h}{X}{\R_+}$ with respect to a probability measure $\mu$ on a measurable space $X$
\begin{equation}\label{eq.entdef}
   \Ent_\mu f = \int_X f(x)\ln f(x) \dd \mu(x) - \left( \int_X f(x)\dd \mu(x) \right)\ln\left( \int_X f(x)\dd \mu(x) \right).
\end{equation}

We say that a closed subset of $\C^n$ \emph{supports the complex $S$-inequality}, \emph{$S\C$-inequality} for short, if for any \emph{cylinder} $C = \{z \in \C^n \ | \ |z_1| \leq R\}$ we have
\begin{equation}\label{eq.defscineq}
   \nu_n(K) = \nu_n(C) \quad \Longrightarrow \quad \nu_n(tK) \geq \nu_n(tC), \quad \text{for $t \geq 1$}.
\end{equation}
Note that the natural counterpart of $S$-inequality in the complex case is the following conjecture due to Prof. A. Pe\l czy\' nski, which  has already been discussed in \cite{Tko}.
\begin{conj*}\label{conjecture}
   All closed subsets $K$ of $\C^n$ which are \emph{rotationally symmetric}, that is $e^{i\theta}K = K$ for any $\theta \in \R$, support $S\C$-inequality.
\end{conj*}
In the present paper we are interested in the class $\mf{R}$ of all closed sets in $\C^n$ which are \emph{Reinhardt complete}, i.e. along with each point $(z_1, \ldots, z_n)$ such a set contains all points $(w_1, \ldots, w_n)$ for which $|w_k| \leq |z_k|$, $k = 1, \ldots, n$ (consult for instance the textbook \cite[I.1.2, pp. 8--9]{Sh}). The key point is that this class contains all unit balls with respect to unconditional norms on $\C^n$. Recall that a norm $\|\cdot\|$ is said to be \emph{unconditional} if $\|(e^{i\theta_1}z_1, \ldots, e^{i\theta_n} z_n)\| = \|z\|$ for all $z \in \C^n$ and $\theta_1, \ldots, \theta_n \in \R$. 

The goal is to prove that all sets from the class $\mf{R}$ support $S\C$-inequality. Now we establish some general yet simple observations which allows us to reduce the problem to a one-dimensional entropy inequality.
\begin{prop}\label{prop.derivative}
   A closed subset $K$ of $\C^n$ supports $S\C$-inequality if for any cylinder $C$ we have
   \begin{equation}\label{eq.derivative}
   \nu_n(K) = \nu_n(C) \quad \Longrightarrow \quad \frac{\dd}{\dd t}\nu_n(tK)\bigg|_{t=1} \geq \frac{\dd}{\dd t}\nu_n(tC)\bigg|_{t=1}.
	\end{equation}
\end{prop}
The proof is essentially given in \cite[Lemma 1]{KS}, so we skip it. For any closed set $A$ the derivative of the function $t \mapsto \nu_n(tA)$ is easy to compute. Indeed,
\begin{align*}
      \frac{\dd}{\dd t}\nu_n(tA)\bigg|_{t=1} &= \frac{\dd}{\dd t} \int_{tA}e^{-|z|^2/2} \dd z \bigg|_{t=1} = \frac{\dd}{\dd t} \int_{A} t^{2n}e^{-t^2|w|^2/2} \dd w \bigg|_{t=1} \\
      &= 2n\nu_n(A) - \int_A |z|^2 \dd \nu_n(z).
\end{align*}
Moreover, the integral of $|z|^2$ over a cylinder $C$ may be expressed explicitly in terms of the measure $\nu_n(C)$. Namely,
\[
	\int_C |z|^2 \dd \nu_n(z) = 2(1- \nu_n(C))\ln\left( 1 - \nu_n(C) \right) + 2n\nu_n(C).
\]
Combining these two remarks with the preceding proposition we obtain an equivalent formulation of the problem..
\begin{prop}\label{prop.entopy}
   A closed subset $K$ of $\C^n$ supports $S\C$-inequality if and only if
   \begin{equation}\label{eq.entropy}
      \int_K |z|^2 \dd \nu_n(z) \leq 2n\nu_n(K) + 2(1 - \nu_n(K))\ln\left( 1 - \nu_n(K) \right).
   \end{equation}
\end{prop}

\section{Main result}\label{sec.res}

We aim at proving the aforementioned main result, which reads as follows
\begin{thm}\label{thm.main}
   Any set from the class $\mf{R}$ supports $S\C$-inequality.
\end{thm}
We begin with a one-dimensional entropy inequality.
\begin{lm}\label{lm.1dimentropyineq}
	Let $\mu$ be a Borel probability measure on $\R_+$ and suppose $\fun{f}{\R_+}{\R_+}$ is a bounded and non-decreasing function. Then
   \begin{equation}\label{eq.1dimentropyineq}
      \Ent_\mu f \leq -\int_{\R_+} f(x)\bigg( 1 + \ln \mu\left( (x, \infty) \right)\bigg) \dd \mu(x).
   \end{equation}
\end{lm}
\begin{proof}
Using homogeneity of both sides of \eqref{eq.1dimentropyineq}, without loss of generality, we can assume that $\int_{\R_+} f \dd \mu = 1$. Then we may rewrite the assertion of the lemma as follows
\[
	\int_{\R_+} \ln \Bigg( f(x)\int_{(x, \infty)} \dd \mu(t) \Bigg) f(x) \dd \mu(x) \leq -1.
\]
Introduce the probability measure $\nu$ on $\R_+$ with the density $f$ with respect to $\mu$. Thanks to monotonicity of $f$ we might estimate the left hand side of the last inequality by
\[
	\int_{\R_+} \ln \bigg( \nu\left( (x, \infty) \right)\bigg) \dd \nu(x) = - \int_0^\infty \int_0^1 \frac{\dd u}{u} \1_{\{u \geq \nu\left( (x, \infty) \right)\}}(u, x) \dd \nu(x).
\]
Define the function
\[
	H(y) := \inf \left\{ t \ | \ \nu\left( (t, \infty) \right) \leq y \right\},
\]
which is the \emph{inverse} tail function, and observe that
\[
	\{(u, x) \ | \ u \geq \nu\left( (x, \infty) \right)\} \supset \{(u, x) \ | \ H(u) \leq x\},
\]
as $u \geq \nu\left( (H(u), \infty) \right) \geq T(x)$. This leads to
\begin{align*}
   - \int_0^\infty \int_0^1 \frac{\dd u}{u} \1_{\{u \geq \nu\left( (x, \infty) \right)\}}(u, x) \dd \nu(x) &\leq - \int_0^\infty \int_0^1 \frac{\dd u}{u} \1_{\{H(u) \leq x\}}(u, x) \dd \nu(x) \\
   &= - \int_0^1 \nu\left( [H(u), \infty) \right) \frac{\dd u}{u}.
\end{align*}
Since $u \leq \nu\left( [H(u), \infty) \right)$, we finally get the desired estimation.
\end{proof}
Now, for a certain class of functions, we establish the multidimensional version of inequality \eqref{eq.1dimentropyineq}. For the simplicity, we formulate this result for the Gaussian measure.
\begin{lm}\label{lm.multidimentropyineq}
   Let $\fun{g}{\C^n}{\R_+}$ be a bounded function satisfying
   \begin{enumerate}[1)]
      \item\label{rotsym} $g((e^{i\theta_1}z_1, \ldots, e^{i\theta_n} z_n)) = g(z)$ for any $z \in \C^n$ and $\theta_1, \ldots, \theta_n \in \R$,
      \item\label{monot} for any $w, z \in \C^n$ the condition $|w_k| \leq |z_k|$, $k = 1, \ldots, n$ implies $g(w) \leq g(z)$.
   \end{enumerate}
   Then
   \begin{equation}\label{eq.multidimentropyineq}
      \Ent_{\nu_n} g \leq \int_{\C^n} g(z)\left( \frac{|z|^2}{2} - n \right) \dd \nu_n(z).
   \end{equation}
\end{lm}
\begin{proof}
One piece of notation: for a fixed vector $r = (r_1, \ldots, r_n) \in (\R_+)^n$ we denote $r^k = (r_1, \ldots, r_{k-1}, r_{k+1}, \ldots, r_n) \in (\R_+)^{n-1}$, and then define the functions  
\[
	g_k^{r^k}(x) = g(r_1, \ldots, r_{k-1}, x, r_{k+1}, \ldots, r_n), \qquad k = 1, \ldots, n.
\]

Notice that for a function $\fun{h}{\C}{\R_+}$ obeying the property \ref{rotsym}) we get
\[
	\int_\C h(z)\dd \nu_1(z) = \frac{1}{2\pi}\int_0^{2\pi} \int_0^\infty h(re^{i\theta})e^{-r^2/2}r\dd r \dd \theta = \int_0^\infty h(r) \dd \mu(r),
\]
where $\mu$ denotes the probability measure on $\R_+$ with the density at $r$ given by $re^{-r^2/2}$. Therefore
\begin{align*}
   \int_{\C^n} g(z)\left( \frac{|z|^2}{2} - n \right) \dd \nu_n(z) &= \int_{(\R_+)^n} g(r) \left( \frac{\sum_{k=1}^n r_k^2}{2} - n \right) \dd \mu^{\otimes n}(r) \\
   &= \int_{(\R_+)^n} \sum_{k=1}^n \Bigg[\int_{\R_+} g_j^{r^j}(x) \left( \frac{x^2}{2} - 1 \right) \dd \mu(x)\Bigg]  \dd \mu^{\otimes n}(r).
\end{align*}
Applying  Lemma \ref{lm.1dimentropyineq} for the function $g_j^{r^j}$ and the measure $\mu$ we obtain the estimation
\begin{align*}
   \int_{\C^n} g(z)\left( \frac{|z|^2}{2} - n \right) \dd \nu_n(z) &\geq \int_{(\R_+)^n} \sum_{k=1}^n  \Ent_{\mu} g_j^{r^j} \dd \mu^{\otimes n}(r) \\ 
   &\geq \Ent_{\mu^{\otimes n}} g = \Ent_{\nu_n} g,
\end{align*}
where the last inequality follows from subadditivity of entropy (for example see \cite[Proposition 5.6]{Led}).
\end{proof}

\begin{proof}[Proof of Theorem \ref{thm.main}]
Fix $K \in \mf{R}$. In order to show \eqref{eq.entropy} we introduce the function $g(z) = 1 - \1_K(z)$. We adopt the standard convention that $0\ln 0 = 0$, hence the desired inequality is equivalent to \eqref{eq.multidimentropyineq}. Thus the application of Lemma \ref{lm.multidimentropyineq} for the function $g$ finishes the proof.
\end{proof}

Theorem \ref{thm.main} immediately implies that the Cartesian products of cylinders support $S\C$-inequality. As a consequence, $S\C$-inequality possesses a tensorization property.
\begin{cor}\label{cor.product}
   Assume sets $K_1 \subset \C^{n_1}, \ldots, K_\ell \subset \C^{n_\ell}$ support $S\C$-inequality. Then the set $K_1 \times \ldots \times K_\ell$ also supports $S\C$-inequality.
\end{cor}

\section*{Acknowledgements}

We would like to thank R. Adamczak for his remark regarding Lemma \ref{lm.1dimentropyineq}, which led to the present general formulation.

The work was done while the second named author was participating in The Kupcinet-Getz International Summer Science School at the Weizmann Institute of Science in Rehovot, Israel.

\appendix

\section{$S$-inequality for the exponential measure in the unconditional case}

Let $\lambda$ be the symmetric exponential measure on $\R$, i.e.
\[
	\dd \lambda (x) = \frac{1}{2}e^{-|x|}\dd x, \qquad x \in \R,
\]
and let $\lambda_n = \lambda \otimes \ldots \otimes \lambda$ be the standard exponential measure on $\R^n$, i.e.
\[
	\dd \lambda (x) = \frac{1}{2^n}e^{-|x|_1} \dd x, \qquad x \in \R^n,
\]
where we denote $|(x_1, \ldots, x_n)|_1 = \sum_{i=1}^n |x_i|$.
It has been recently noticed that the technique of the paper applies also to the $S$-inequality for the measure $\lambda_n$.  The result reads as follows
\begin{thm}\label{thm.mainexp}
	For any closed convex subset $K \subset \R^n$ which is unconditional, i.e. $(\epsilon_1 x_1, \ldots, \epsilon_n x_n) \in K$ whenever $(x_1, \ldots, x_n) \in K$ and $\epsilon_1, \ldots, \epsilon_n \in \{-1, 1\}$, and for any strip $P = \{x \in \R^n \ | \ |x_1| \leq p\}$, $p \geq 0$, we have
	\begin{equation}\label{eq.sineqexp}
		\lambda_n(K) = \lambda_n(P) \quad \Longrightarrow \quad \forall t \geq 1 \ \lambda_n(tK) \geq \lambda_n(tP),
	\end{equation}
	and, equivalently,
	\begin{equation}\label{eq.sineqexp'}
		\lambda_n(K) = \lambda_n(P) \quad \Longrightarrow \quad \forall t \leq 1 \ \lambda_n(tK) \leq \lambda_n(tP).
	\end{equation}
\end{thm}
\begin{proof}
The equivalence between \eqref{eq.sineqexp} and \eqref{eq.sineqexp'} is straightforward. For instance, assume the latter does not hold. Then, there is $t_0 < 1$ such that $\lambda_n(t_0K) > \lambda_n(t_0P)$. So, we can find $s_0 < 1$ for which $\lambda_n(s_0t_0K) = \lambda_n(s_0t_0P)$. Using \eqref{eq.sineqexp} we get a contradiction
\[
	\lambda_n(K) > \lambda_n(s_0K) = \lambda_n\left( \frac{1}{t_0}\left( s_0t_0K \right) \right) \geq \lambda_n\left( \frac{1}{t_0}\left( t_0P \right) \right) = \lambda_n(P) = \lambda_n(K).
\]

We essentially follow the proof of Theorem \ref{thm.main}. Hence, first of all, as in Proposition \ref{prop.derivative}, we notice that we might equivalently prove that
\begin{equation}\label{eq.sineqexpviaderiv}
		\lambda_n(K) = \lambda_n(P) \quad \Longrightarrow \quad \frac{\dd}{\dd t}\lambda_n(tK)\bigg|_{t=1} \geq \frac{\dd}{\dd t}\lambda_n(tC)\bigg|_{t=1}.
\end{equation}
Again, a trivial computation shows that
\begin{equation}\label{eq.derivat1exp}
	\frac{\dd}{\dd t}\lambda_n(tK)\bigg|_{t=1} = n\lambda_n(K) - \int_K |x|_1 \dd \lambda_n(x).
\end{equation}
Thus, we would like to prove that
\begin{equation}\label{eq.sineqexpviamoments}
	\lambda_n(K) = \lambda_n(P) \quad \Longrightarrow \quad \int_K |x|_1 \dd \lambda_n(x) \leq \int_P |x|_1 \dd \lambda_n(x).
\end{equation}
One another easy computation yields
\[
	\int_P |x|_1 \dd \lambda_n(x) = n(1 - e^{-p}) - pe^{-p}.
\]
But $\lambda_n(K) = \lambda_n(P) = 1 - e^{-p}$, so we get
\[
	\int_P |x|_1 \dd \lambda_n(x) = n - n\lambda_n(K') + \lambda_n(K')\ln \lambda_n(K'),
\]
where $K' = \R^n \setminus K$. Since $\int_{\R^n} |x|_1 \dd \lambda_n(x) = n$, \eqref{eq.sineqexpviamoments} is equivalent to
\[
	-\lambda_n(K')\ln \lambda_n(K') \leq \int_{K'} (|x|_1 - n) \dd \lambda_n(x).
\]
Now we introduce the function $g(x) = \1_{K'}(x)$. Then the above inequality may be rewritten to
\begin{equation}\label{eq.sineqexpviag}
	\Ent_{\lambda_n} g \leq \int_{\R^n} g(x)(|x|_1 - n) \dd \lambda_n(x).
\end{equation}
Thanks to unconditionality of $K$ the function $g$ is even with respect to each coordinate. Using in addition convexity, we check that it is nondecreasing with respect to each coordinate. These properties as well as certain one-dimensional inequality, which is deduced from Lemma \ref{lm.1dimentropyineq} for the measure with the density $e^{-x}$ on $\R_+$, allow us to prove inequality \eqref{eq.sineqexpviag} in the same manner as in Lemma \ref{lm.multidimentropyineq}.
\end{proof}

Following the method of \cite[Corollary 3]{LO} we obtain the result concerning the comparison of moments.
\begin{cor}\label{cor.expmoments}
	Let $\|\cdot\|$ be a norm on $\R^n$ which is unconditional, i.e.
	\[
		|(\epsilon_1 x_1, \ldots, \epsilon_n x_n)\| = \|(x_1, \ldots, x_n)\|,
	\]
	for any $x_j \in \R$ and $\epsilon_j \in \{-1, 1\}$. Then for $p \geq q > 0$
	\begin{equation}\label{eq.expmoments}
		\left( \int_{\R^n} \|x\|^p \dd \lambda_n(x) \right)^{1/p} \leq C_{p,q} \left( \int_{\R^n} \|x\| ^q \dd \lambda_n(x) \right)^{1/q},
	\end{equation}
	where the constant $C_{p,q} = \left( \int_\R |x|^p \dd \lambda(x) \right)^{1/p} / \left( \int_\R |x|^q \dd \lambda(x) \right)^{1/q}$ is the best possible.
\end{cor}
\begin{proof}
	The proof hinges on the fact that a ball $K = \{x \in \R^n \ \|x\| \leq t\}$ with respect to the norm $\|\cdot\|$ is a closed convex unconditional set, so that Theorem \ref{thm.mainexp} can be applied.
\end{proof}

\begin{remark}
	It has been tempting to see how the proof might have worked for product measures with the density $C_\rho^n e^{-\sum_{i=1}^n \rho(|x_i|)}$, where $\rho$ is, e.g., positive convex and increasing function on $(0, \infty)$ and $\rho(0) = 0$. The only problem is that this in not entropy but another functional, constructed out of $\rho$, that appears in \eqref{eq.sineqexpviag}. In the case of functions $\rho(x) = x^p$, $p > 1$, we did check that Corollary 3 of \cite{LO2} (consult there Definition 4 as well) does not give that such functionals possess desired properties such as subadditivity. From this point of view the exponential measure seems to be quite exceptional.
\end{remark}

\noindent Piotr Nayar \\
\noindent Institute of Mathematics, University of Warsaw, \\
\noindent Banacha 2, \\
\noindent 02-097 Warszawa, Poland. \\
\noindent \texttt{nayar@mimuw.edu.pl}

\vspace{1em}

\noindent Tomasz Tkocz \\
\noindent Institute of Mathematics, University of Warsaw, \\
\noindent Banacha 2, \\
\noindent 02-097 Warszawa, Poland. \\
\noindent \texttt{tkocz@mimuw.edu.pl}


\begin{thebibliography}{9}


\bibitem[KS]{KS} S. Kwapie\'n, J. Sawa, On some conjecture concerning Gaussian measures of dilatations of convex symmetric sets, {\em Studia Math.}, \textbf{105} (1993), no. 2, 173--187. MR1226627 (94g:60011)

\bibitem[Lat]{Lat} R. Lata\l a, On some inequalities for Gaussian measures (English summary), \emph{Proceedings of the International Congress of Mathematicians, Vol. II (Beijing, 2002)},  813–822, \emph{Higher Ed. Press, Beijing}, 2002. MR1957087 (2004b:60055)

\bibitem[LO1]{LO} R. Lata\l a, K. Oleszkiewicz, Gaussian measures of dilatations of convex symmetric sets, {\em Ann. Probab.} {\bf 27} (1999), no. 4, 1922--1938. MR1742894 (2000k:60062)

\bibitem[LO2]{LO2} R. Lata\l a, K. Oleszkiewicz, Between Sobolev and Poincar\'e, {\em Geometric aspects of functional analysis}, 147--168, {\em Lecture Notes in Math.}, 1745, {\em Springer, Berlin}, 2000. MR1796718 (2002b:60025)


\bibitem[Led]{Led} M. Ledoux, The concentration of measure phenomenon. Mathematical Surveys and Monographs, 89. {\em American Mathematical Society, Providence, RI}, 2001. MR1849347 (2003k:28019)

\bibitem[Sh]{Sh} B. V. Shabat, Introduction to complex analysis. Part II. Functions of several variables. Translated from the third (1985) Russian edition by J. S. Joel. Translations of Mathematical Monographs, 110. {\em American Mathematical Society, Providence, RI}, 1992. MR1192135 (93g:32001)

\bibitem[Tko]{Tko} T. Tkocz, Gaussian measures of dilations of convex rotationally symmetric sets in $\C^n$, {\em Electron. Commun. Probab.} {\bf 16} (2011), 38--49. MR2763527 (Review)


\end{thebibliography}
\end{document}